\documentclass[final,3p,times]{elsarticle}

\usepackage{amssymb}
\usepackage{amsthm}
\usepackage{graphicx}
\usepackage{graphics} 
\usepackage{amsmath}
\usepackage{verbatim}
\usepackage{psfrag}
\usepackage{tikz}
\usepackage{epstopdf}
\usepackage{algorithmic}
\usepackage[linesnumbered,ruled,vlined]{algorithm2e}

\journal{arXiv}

\begin{document}

\newtheorem{theorem}{Theorem}
\newtheorem{lemma}{Lemma}
\newtheorem{proposition}{Proposition}
\newtheorem{corollary}{Corollary}
\newtheorem{definition}{Definition}
\newtheorem{remark}{Remark}

\def\fprod#1{\left\langle#1\right\rangle}

\begin{frontmatter}

\title{\LARGE \bf
Non-Concave Network Utility Maximization in Connectionless Networks: A Fully Distributed Traffic Allocation Algorithm\tnoteref{acknowledgement}
}

\author{Jingyao Wang\corref{cor1}\fnref{address_wang} }
\author{Mahmoud Ashour\fnref{address_mahmound}}
\author{ Constantino Lagoa\fnref{address_mahmound}}
\author{Necdet Aybat\fnref{address_necdet}}
\author{Hao Che\fnref{address_hao}} 
\author{Zhisheng Duan\fnref{address_wang}}
 \fntext[address_wang]{Jingyao Wang and Zhisheng Duan are with the Department of Mechanics and Engineering Science, Peking University, Beijing 100871, China, (yayale.8@163.com, duanzs@pku.edu.cn).}

\fntext[address_mahmound]{Mahmoud Ashour and Constantino Lagoa are with the Department of Electrical Engineering and Computer Science, Pennsylvania State University, University Park, PA 16802, USA, (mma240@psu.edu, lagoa@engr.psu.edu).}

\fntext[address_necdet]{Necdet Aybat is with the Department of Industrial Engineering, Pennsylvania State University,
      University Park, PA 16802, USA, (nsa10@psu.edu).}
        
\fntext[address_hao]{Hao Che is with the Department of Computer Science and Engineering, University of Texas at Arlington, Arlington, TX 769019, USA, (hche@cse.uta.edu).}

\tnotetext[acknowledgement]{ This work was partially supported by NSF grants CNS-1329422, CMMI-1635106,
FCC-1629625, NNSF of China grants 61673026 and the China Scholarship Council.}

\begin{abstract}
This paper considers the optimization-based traffic allocation problem among multiple end points in connectionless networks. The network utility function is modeled as a non-concave function, since it is the best description of the quality of service perceived by users with inelastic applications, such as video and audio streaming. However, the resulting non-convex optimization problem, is challenging and requires new analysis and solution techniques. To overcome these challenges, we first propose a hierarchy of problems whose optimal value converges to the optimal value of the non-convex optimization problem as the number of moments tends to infinity.
From this hierarchy of problems, we obtain a convex relaxation of the original non-convex optimization problem by considering truncated moment sequences. For solving the convex relaxation, we propose a fully distributed iterative algorithm, which enables each node to adjust its date allocation/ rate adaption among any given set of next hops solely based on information from the neighboring nodes. Moreover, the proposed traffic allocation algorithm converges to the optimal value of the convex relaxation at a $O(1/K)$ rate, where $K$ is the iteration counter, with a bounded optimality. At the end of this paper, we perform numerical simulations to demonstrate the soundness of the developed algorithm.
\end{abstract}

\end{frontmatter}

\section{INTRODUCTION}
Applications and services supported by modern communication networks have diverse requirements, e.g., high throughput and low latency. Traffic engineering (TE) has long been used to optimize the utilization of the limited network resources so that such requirements are fulfilled. This entails developing data rate allocation algorithms and congestion control protocols capable of maximizing a given network utility subject to network resource constraints \cite{Kelly_JORS_1998}. Many problems of recent interest arising in diverse fields can be cast as an optimization problem, and network utility maximization (NUM) is no different.

In large-scale networks, the size of the optimization problems rapidly increases as the number of nodes and links increase. This stimulates the necessity of developing decentralized control algorithms capable of decomposing the high-dimensional problem into separate moderate-size subproblems that can be solved independently and locally at various network nodes. The main idea behind such decentralized control algorithms is to distribute the computations required for the solution of the optimization problem among various nodes \cite{lagoa_2004_N}-\cite{beck2013optimal}. This approach exploits local information available at each node. Nevertheless, information exchange among different nodes is inevitable since distinct data flows share the same network resources. Therefore, distributed optimization approaches not only aim at decomposing the problem, but also minimizing the communication overhead.

In the benchmark work by Kelly et. al. \cite{Kelly_JORS_1998}, the optimization of the utility of a large-scale broadband network with limited bandwidth resources is considered. The authors propose two classes of rate control algorithms by casting the NUM problem in both primal and dual forms. In \cite{lagoa_2004_N}, a family of decentralized sending rate control laws are proposed to steer the traffic allocation to an optimal operating point while avoiding congestion. A non-linear control theoretic approach is employed in \cite{lagoa_hop} to derive adaptation laws that enable each node to independently distribute its traffic optimally among any given set of next hops. More recently, reference \cite{beck2013optimal} considers the NUM, derives its dual problem, and uses a distributed gradient-based approach for its solution. A similar approach appears in \cite{nekouei}. In spite of the existence of a relatively dense literature on NUM, most available results consider only the optimization of concave utility functions. However, it has been shown that the reward experienced by the users of real-time applications, such as video and audio streaming, cannot be accurately modeled using concave functions. Reference \cite{Yin_2015_SC} shows that the video quality perceived by users on a mobile device is a non-decreasing and step-like function with respect to the data rate, because users have almost similar quality of experience on $3$ Mbps and $1$ Mbps \cite{Yin_2015_SC}. This observation motivates considering the optimization of non-concave network utility functions, which constitutes a main focus of this paper.

Non-concave NUM is a non-convex optimization problem; hence, it is difficult to solve. Nevertheless, there exist some attempts in the literature for deriving algorithms that provide near-optimal solutions. Reference \cite{Fazel_CDC_2005} develops a centralized algorithm that solves the NUM problem with polynomial utilities. Reference \cite{Hande_Net_2007} determines the conditions under which the standard distributed dual-based algorithm can still converge to the global optimal solution with non-concave utilities.

This paper develops a distributed iterative algorithm for the optimization of a generalized class of non-concave network utility functions that capture a wide variety of real-world applications. In particular, we focus on connectionless networks, where each node is required to distribute its traffic among a set of next hops without prior arrangement so that the network utility is maximized. We handle the challenge posed by the non-convexity of the optimization problem by developing a sequence of convex relaxations whose solution converges to that of the original problem. We use results on polynomial optimization and moment sequences to derive the convex relaxations \cite{lasserre,laurent}. Furthermore, we propose an iterative primal-dual algorithm \cite{Aybat_arxiv_2016} that enables each node to distribute its traffic among the set of next hops. We emphasize on the distributed nature of the algorithm, where each node uses its local information and need not communicate with other nodes except its direct neighbors.

\section{NOTATION}

Throughout this paper, the traffic flows are assumed to be described by a fluid flow model, and the only resource constraint taken into account is link bandwidth. In the remainder of this paper, call and flow will be used interchangeably.

Let $\mathcal{N}$ denote the set of nodes in the network,
and $\mathcal{L}\subset \mathcal{N} \times \mathcal{N}$ denote the set of links connecting particular pairs of nodes. We assume that each link $l \in \mathcal{L}$ has a finite capacity $c_l>0$. Moreover, let $\mathcal{S}\triangleq \{ s_1,s_2,\ldots,s_n\}$ and $\mathcal{D}\triangleq \{ d_1,d_2,\ldots,d_n\}$ denote respectively the set of source nodes and the set of destination nodes contained in $\mathcal{N}$ such that $\mathcal{S} \cap \mathcal{D}= \emptyset$. The intended destination for each source node $s_i$ is $d_i$ for $i\in \mathcal{I}\triangleq \{1,\ldots,n\}$, i.e., without loss of generality, we assume that there is a one-to-one correspondence between  $\mathcal{S}$ and $\mathcal{D}$, and $\mathcal{I}$ denotes the set of different flow (call) types in the network. Given source node $s\in\mathcal{S}$, let $\mathcal{L}_{s}$ denote the set of links connected to it. Let the sending data rate through link $l\in\mathcal{L}_{s}$ be $x_{s,l}^{\text{out}}$, and all such sending data rates be $\mathbf{x}_{s}^{\text{out}}\triangleq [x_{s,l}^{\text{out}}]_{l\in\mathcal{L}_{s} }$. We define the aggregate sending data rate of $s\in\mathcal{S}$ be denoted by $r_{s}\triangleq \sum_{l\in\mathcal{L}_{s} }x_{s,l}^{\text{out}} $. Also, let $\mathcal{B}\triangleq \mathcal{N} \setminus  (\mathcal{S} \cup \mathcal{D})= \{b_1,b_2,\ldots,b_m\}$ denote the set of forwarding nodes contained in $\mathcal{N}$. Given $b\in\mathcal{B} $, let $\mathcal{I}_b$ be the set of flows visiting node $b$, and $\mathcal{L}_b\subseteq \mathcal{L}$ denote the set of links connected to it. Suppose $\mathcal{L}_{b,i}^{\text{out}}\subseteq \mathcal{L}_b$ denote the set of outgoing links from $b$ associated with calls (flows) of type $i\in \mathcal{I}_b$. Similarly, let $\mathcal{L}_{b,i}^{\text{in}} \subset \mathcal{L}_b$ denote the set of incoming links to $b$ associated with calls (flows) of type $i \in\mathcal{I}_b$. Furthermore, given $b\in\mathcal{B} $, for each $i\in \mathcal{I}_b$ and $l\in \mathcal{L}_{b,i}^{\text{out}}$, let $x_{i,b,l}^{\text{out}}$ denote the data rate of call type $i \in\mathcal{I}_b$, associated with  $s_i$ and $d_i$, forwarded from node $b$ through link $l\in\mathcal{L}_{b,i}^{\text{out}} $.
The above notation is exemplified in Fig.~1 for the case of allocating flows associated with two source nodes, $s_1$ and $s_2$, and two destination nodes, $d_1$ and $d_2$.

Given $b\in\mathcal{B} $ and $l\in \mathcal{L}_b$, 
 let $\mathcal{I}_{b,l}^{\text{in}} \subset \mathcal{I}$ be the set of call types forwarded to node $b$ through link $l$, and $\mathcal{I}_{b,l}^{\text{out}}\subseteq \mathcal{I}_b$ be the set of call types forwarded from node $b$ through link $l$. Moreover, given node $b\in\mathcal{B} $ and link $l\in \mathcal{L}_b$, let $e_l(b)$ denote the adjacent node to $b$ through link $l$.
We summarize all the notation for the communication network in Table I for the convenience of the reader.

Now, given node $b\in\mathcal{B} $, let the vector containing all flow rates departing from node $b$ through link $l\in\mathcal{L}_{b} $ be denoted by $\mathbf{x}_{b,l}^{\text{out}}\triangleq [x_{i,b,l}^{\text{out}}]_{i\in\mathcal{I}_{b,l}^{\text{out}}}\in \text{R}_+^{|\mathcal{I}_{b,l}^{\text{out}}|}$, where $|.|$ denotes the cardinality of a set.

Given node $b\in\mathcal{B} $ and $l\in \mathcal{L}_b$, let $\mathbf{1}_{b,l}\in \text{R}^{1\times |\mathcal{I}_{b,l}^{\text{out}}|}$ be the row vector with all elements equal to $1$. In a similar way, let $\delta_{b,l}\in \text{R}^{1\times |\mathcal{I}_{b,l}^{\text{in}}|}$ be the row vector with all elements equal to $1$ if link $l$ is bidirectional, and $0$ otherwise.

Also, let $\|.\|$ denote the Euclidean norm. Given a convex set $\mathcal{A}$, let $\mathit{I}_{\mathcal{A}}(.)$ denote the indicator function of $\mathcal{A}$, i.e., $\mathit{I}_{\mathcal{A}}(\omega)=0$ for $\omega \in \mathcal{A}$ and equal to $+ \infty$ otherwise, and let $\mathit{P}_{\mathcal{A}}(\omega)\triangleq  \text{argmin} \{ \|\upsilon-\omega\|: \upsilon \in  \mathcal{A} \}$ denote the projection onto $\mathcal{A}$. Given a closed convex set $\mathcal{A}$, we define the distance function as $d_{\mathcal{A}}(\omega)\triangleq \|\mathit{P}_{\mathcal{A}}(\omega)-\omega\| $. Also, $\mathbf{I}_n$ is the $n\times n$ identity matrix.

\begin{figure}
\centering
\includegraphics[width=1\columnwidth, height=0.45\columnwidth]{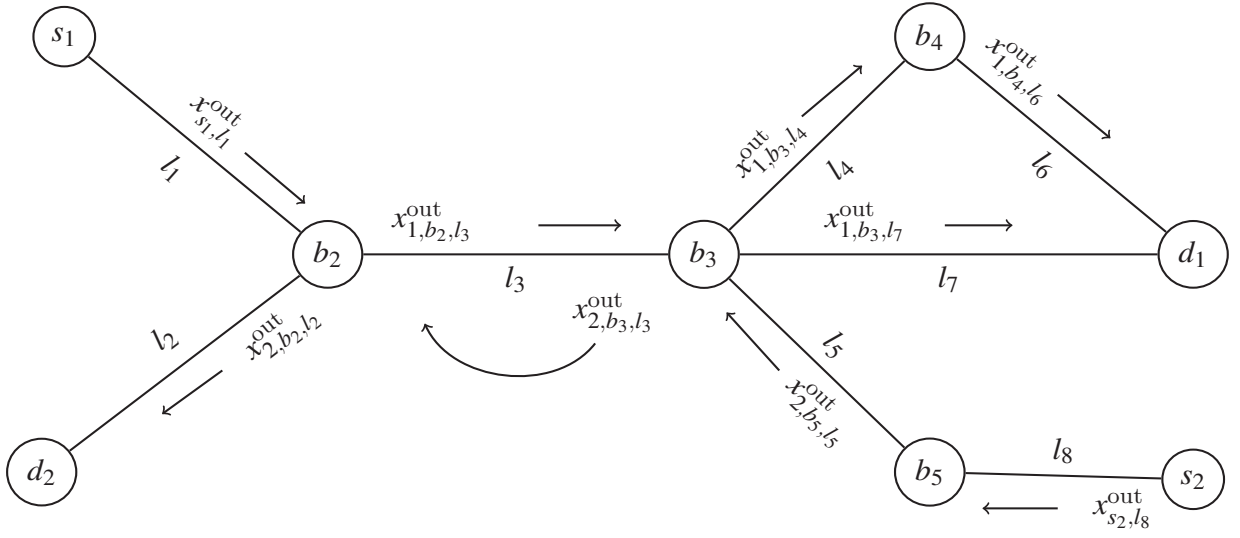} \\
\caption{Notation example.}
\label{fig_3}
\end{figure}

\begin{table}
\centering
\caption{LIST OF NOTATION}
\label{table:notation}
\resizebox{0.7\textwidth}{!}{
 \begin{tabular}{| c | l |}
 \hline
 Notation & Desciption \\
 \hline
 $\mathcal{N}$  & The set of nodes in the network. \\[1pt]
  $\mathcal{S}$ & The set of source nodes.\\[1pt]
  $\mathcal{B}$  & The set of forwarding nodes. \\[1pt]
  $e_l(b)$ & The node connected to node $b$ through link $l$. \\[3pt]
  $\mathcal{L}_{s}$ ($\mathcal{L}_b$)& The set of links connected to node $s$ (node $b$). \\[1pt]
 $\mathcal{L}_{b,i}^{\text{out}} (\mathcal{L}_{b,i}^{\text{in}})$& The set of outgoing (incoming) links from (at) node $b$\\[1pt]
 &for flows of type ${i}\in \mathcal{I}$. \\[1pt]
   $\mathcal{I}$ & The set of different flow types.\\[1pt]
  $\mathcal{I}_b$ & The set of flows visiting node $b$. \\[1pt]
  $\mathcal{I}_{b,l}^{\text{out}} (\mathcal{I}_{b,l}^{\text{in}})$& The set of flows forwarded from (to) node $b$ through link $l$. \\[3pt]
  $r_{s}$ & The aggregate data rate of source node $s$.\\[1pt]
  $x_{s,l}^{\text{out}}$& The sending data rate of source node $s\in \mathcal{I}$ through link $l$. \\[1pt]
  $\mathbf{x}_{s}^{\text{out}}$& The vector consisting of $x_{s,l}^{\text{out}}$ for each link $l\in\mathcal{L}_{s} $. \\[1pt]
  $x_{i,b,l}^{\text{out}}$ & The data rate of flows belonging to source node $s_i$ forwarded from node $b$ \\[1pt]
  & through link $l$. \\[1pt]
  $\mathbf{x}_{b,l}^{\text{out}}$ & The vector consisting of $x_{i,b,l}^{\text{out}}$ for each type of flow $i\in \mathcal{I}_{b,l}^{\text{out}}$.\\[1pt]
 \\[1pt]
  \hline
\end{tabular}
}
\end{table}

\section{PROBLEM FORMULATION}

Consider a communication network consisting of a set of source nodes $\mathcal{S}$. Each source node $s\in \mathcal{S}$ has a local utility function $U_s(r_{s}):\mathrm{R}_{+}\to \mathrm{R}_{+}$ of its sending data rate $r_{s}$. For a fixed order $\ell >0$, the utility function is defined as a general non-concave polynomial-like function in the form
\begin{equation}\label{local_utility}
U_s(r_{s})\triangleq \sum \limits_{j=0}^{\ell} p_{s,j} (r_{s})^{j/\ell}.
\end{equation}
This particular form of objective functions is so flexible that it can be used to approximate a wide variety of functions arising in practical applications such as step functions for the video streaming case \cite{nekouei}.

The objective of this paper is to design a data rate allocation algorithm for the communication network such that the utilization of resources is maximized, while satisfying the network resource constraints. The network resource constraints considered in this paper include link capacity constraints, Minimum Rate Guaranteed and Upper Bounded Rate Service (MRGUBRS) requirements, and flow conservation constraints through nodes.

More precisely, for any link $l\in \mathcal{L}$, the aggregated flows going through this link should not exceed the link capacity. For example, in Fig.~1, the bidirectional link $l_3$ is shared by flows belonging to two source nodes. The data rates $x_{1,b_2,l_3}^\text{out}$ and $x_{2,b_3,l_3}^\text{out}$ going through this link should satisfy that
\begin{equation}\label{capacity_example}
x_{1,b_2,l_3}^\text{out} + x_{2,b_3,l_3}^\text{out} \leq c_{l_3}.
\end{equation}
For the unidirectional link $l_2$,
node $b_2$ forwards data rate $x_{2,b_2,l_2}^\text{out}$ through this link. Then, $x_{2,b_2,l_2}^\text{out}$ is upper bounded by $c_{l_2}$.

Given flows of type $i\in \mathcal{I}$, recall that flows of type $i\in \mathcal{I}$ is associated with source/destination pairs $s_i/d_i$.   For fixed link $l\in \mathcal{L}_{s_i}$, the corresponding data rate $x_{i,l}^\text{out}$ is determined at source node $s_i \in \mathcal{S}$ and multiple paths are available for transporting these flows.
More precisely, each node on these paths divide incoming traffic into available links by striving to conserve the flows belonging to each source node (i.e., aims at no losses) and to avoid link congestion. In Fig.~1, node $b_3$ tries to satisfy
\begin{equation}\label{flow_conservation_node_example_one}
x_{1,b_2,l_3}^\text{out}=x_{1,b_3,l_4}^\text{out}+x_{1,b_3,l_7}^\text{out}.
\end{equation}

Finally, flows belonging to each source node $s\in \mathcal{S}$ is assumed to be of the MRGUBS category, i.e., for some $0< \xi_s<\zeta_s$ and $s\in \mathcal{S}$,
\begin{equation}\label{mrgubrs_example}
\xi_s\leq r_{s}\leq \zeta_s.
\end{equation}

Now, considering the above constrains and assumptions, we can formulate the problem of optimal
traffic allocation as follows:
\begin{equation}\label{objective_func_one}
\text{maximize} \sum\limits_{s\in \mathcal{S}} U_{s}(r_{s}),
\end{equation}
subject to the network capacity constraints \footnote{Note that the formulation in this paper allows for the existence of bidirectional links.}
\begin{equation*}
\sum\limits_{i \in \mathcal{I}_{b,l}^{\text{out}}} x_{i,b,l}^{\text{out}}+\sum\limits_{i \in \mathcal{I}_{b,l}^{\text{in}}}x_{i,e_l(b),l}^{\text{out}} \leq c_l,\ l \in \mathcal{L}_b,\ b \in \mathcal{B},
\end{equation*}
the flow conservation constraints at each node
\begin{equation*}\label{fccn_v_one}
\sum\limits_{l \in \mathcal{L}_{b,i}^{\text{in}}} x_{i,e_l(b),l}^{\text{out}}- \sum\limits_{\tilde{l} \in \mathcal{L}_{b,i}^{\text{out}}}x_{i,b,\tilde{l}}^{\text{out}}=0,\ i \in \mathcal{I}_b,\  b \in \mathcal{B},
\end{equation*}
the non-negativity of forwarded data rates constraints
\begin{equation*}\label{non_negative_out_rate_one}
x_{i,b,l}^{\text{out}}\geq 0,\ i\in \mathcal{I}_{b,l}^{\text{out}},\ l\in \mathcal{L}_b,\  b\in \mathcal{B},
\end{equation*}
and the MRGUBS requirements
\begin{equation*}\label{mrgubs_v_one}
( \mathbf{x}_{s}^{\text{out}},\ r_{s})\in \mathcal{X}_{s},\ s\in \mathcal{S},
\end{equation*}
where the set $\mathcal{X}_{s}$ is defined as
\begin{equation*}
\mathcal{X}_{s}\triangleq \bigg\{  ( \mathbf{x}_{s}^{\text{out}},\ r_{s})\in \text{R}_+^{|\mathcal{L}_{s}|}\times \text{R}_+:
\xi_s \leq r_{s} \leq \zeta_s,\ r_{s}= \sum_{l\in\mathcal{L}_{s} } x_{i,l}^{\text{out}}\bigg\}.
\end{equation*}

Most literature in the context of NUM considers maximizing concave diminishing functions.
However,
modern communication networks are dominated by various inelastic applications, such as internet video and audio streaming. Users' satisfaction for these applications cannot be modeled with concave functions. It is better to be described as non-concave functions. For instance, the utility for voice applications is a sigmoidal function \cite{Fazel_CDC_2005}. Thus, we consider users' perceived qualification of Cost of Service (CoS) and model the utility function as a general class of non-concave polynomial functions.
Moreover, the challenges of attempting to solve the resulting traffic allocation problem (\ref{objective_func_one}) are two-fold.
First, the optimization problem obviously constitutes a non-convex problem since its objective function is non-concave. Second, global information on fast timescale events, as required in the above formulation, is not generally available. The latter fact stimulates the necessity of developing a distributed algorithm that converges to the optimal data rate allocation of the non-convex NUM problem.

\section{MAIN RESULTS}
In this section, we present our approach used to overcome the challenges opposed by the non-convexity of the optimization problem. In particular, we first present a convex relaxation to the non-convex NUM problem (\ref{objective_func_one}). This convex relaxation is chosen from a hierarchy of optimization problems whose optimal value converges to the optimal value of problem (\ref{objective_func_one}) as the number of moments tends to infinity. For solving the convex relaxation problem, we propose a distributed primal-dual algorithm (DPDA), which enables all nodes to update their data rate allocation solely using immediate local information.
A salient feature of the proposed algorithm is that the iterate sequence converges to the optimal solution at a $O(1/K)$ rate, where $K$ is the iteration counter, with a bounded optimality.

\subsection{NUM convex relaxation}
The non-convexity of the optimization problem (\ref{objective_func_one}) opposes challenges for us to analysis and solve the traffic allocation problem. However, the following proposition provides a hierarchy of optimization problems whose optimal value converges to the optimal value of the non-convex optimal problem (\ref{objective_func_one}). For solving the traffic allocation problem, we choose a convex one from this hierarchy of problem by truncating the number of moments to the finite case. This proposition is one of the main results of this paper.

\begin{proposition}\label{Proposition1}
The solution of the following optimization problem converges to the solution
of the non-convex NUM problem (\ref{objective_func_one}) with non-concave user utility functions of the form (\ref{local_utility}) as the positive parameter $\alpha \rightarrow \infty$. Moreover, problem (\ref{op_five}) is convex if $\alpha \leq \ell$.
\begin{equation}\label{op_five}
\begin{aligned}
&\underset{ \mathbf{x}}{\text{maximize}}
& & \sum\limits_{s \in \mathcal{S}} \mathbf{p}_{s}^T\mathbf{m}_{s}\\
& \text{subject\ to}\
& & {m}_{s,0}=1,\  s\in \mathcal{S},\\
& & &  \mathbf{M}(0, \alpha,\mathbf{m}_{s})\succeq 0,\  s\in \mathcal{S}, \\
& & & \beta_{s}  \mathbf{M}(0, \alpha-2,\mathbf{m}_{s})- \mathbf{M}(2, \alpha,\mathbf{m}_{s}) \succeq 0, \\
& & & m_{s,j} \leq (r_{s})^{j/\ell},\ j\in \{1,\ldots,\alpha\},\ s \in \mathcal{S},\\
& & & x_{s,l}^{\text{out}} \leq c_l, l \in \mathcal{L}_{s},\  s\in \mathcal{S},\\
& & & \mathbf{1}_{b,l} \mathbf{x}_{b,l}^{\text{out}}+\delta_{b,l}\mathbf{x}_{e_l(b),l}^{\text{out}} \leq c_l,\  l \in \mathcal{L}_b,\ b \in \mathcal{B},\\
& & & B \mathbf{x}=0,\\
& & & ( \mathbf{x}_{s}^{\text{out}},\ r_{s})\in \mathcal{X}_{s},\ s\in \mathcal{S},\\
& & & \mathbf{x}_{b,l}^{\text{out}}\succeq 0,\ l\in \mathcal{L}_b,\  b\in \mathcal{B}.
\end{aligned}
\end{equation}
The objective function is a linear function of variables $\mathbf{m}_{s}=[m_{s,j}]_{j \in \{0,\hdots,\alpha\}}$ with parameters $\mathbf{p}_{s}=[p_{s,j}]_{j \in \{0,\hdots,\alpha\}}$.
The decision variable $\mathbf{x}$ of problem (\ref{op_five}) is a vector consisting of the data rate $x_{s,l}^{\text{out}}$, $r_{s}$ and $\mathbf{m}_{s}$ for each $s\in\mathcal{S}$, and the sending data rate
$\mathbf{x}_{b,l}^{\text{out}},\ b\in\mathcal{B} $ for each $l\in \mathcal{L} $. More precisely, the dimension of vector $\mathbf{x}$ is $\sum_{s\in\mathcal{S}} (|\mathcal{L}_{s} |+\alpha+2)+\sum_{b\in\mathcal{B}}\sum_{i\in\mathcal{I}_{b}}|\mathcal{L}_{b,i}^{\text{out}}|$.
In the constraints,
$B\in \mathrm{R}^{\left(\sum_{b\in \mathcal{B}} |\mathcal{I}_b|\right)\times \left(\sum_{s\in\mathcal{S}} (|\mathcal{L}_{s} |+\alpha+2)+\sum_{b\in\mathcal{B}}\sum_{i\in\mathcal{I}_{b}}|\mathcal{L}_{b,i}^{\text{out}}| \right)}$ denotes the edge-node-like incidence matrix, i.e., the entry $B_{(s,b,l),\omega}$, corresponding to flow-node-link triplet $(s,b,l)\in \mathcal{S}\times \mathcal{B}\times \mathcal{L} $ and $\omega\in \mathbf{x}$, equal to $1$ if the data rate $\omega$ of flows belonging to source node $s$ is forwarded from node $b$ through link $l$, $-1$ if the the data rate $\omega$ is received at node $b$, and $0$ otherwise.
 $\beta_{s} $ is a known upper bound on the aggregate data rate of source $s\in\mathcal{S}$,
and the moment matrices $\mathbf{M}\in \mathrm{R}^{h+1} \times \mathrm{R}^{h+1}$ are of the form
\begin{equation}\label{hankel_matrix_two}
\mathbf{M}(k,k+2h,\mathbf{m}_{s}) =
\left[\begin{smallmatrix}
m_{s,k} & m_{s,k+1} & \ldots &  m_{s,k+h}\\
m_{s,k+1} &  \reflectbox{$\ddots$}  &  \reflectbox{$\ddots$}   & m_{s,k+h+1}\\
\vdots &   \reflectbox{$\ddots$}   & \reflectbox{$\ddots$}   &  \vdots\\
m_{s,k+h} & \ldots &\ldots & m_{s,k+2h}
\end{smallmatrix}\right].
\end{equation}

\end{proposition}

\begin{proof}
The proof is shown in Appendix A.
\end{proof}

Hereafter, we use $\alpha = \ell$.
It is worth mentioning that the result of Proposition $1$ holds for the even order $\ell$. Nonetheless, similar results can be derived for the odd $\ell$, which is omitted for brevity. The proposed problem (\ref{op_five}) constitutes a convex optimization problem, because it maximizes the sum of linear functions subject to convex constraints. Therefore, it can be easily solved if global information is available. Nevertheless, the objective of this paper is to solve this problem in a distributed fashion that leverages per hop information available at each node.

Before moving on, we
introduce some notation that renders the formulation of (\ref{op_five}) conveniently compact.
For every $s \in \mathcal{S}$, let the set $\mathcal{A}_{s}$ be defined as
\begin{equation}\label{set_s}
{\small
\begin{aligned}
& \mathcal{A}_{s}=\{  (\mathbf{x}_{s}^{\text{out}}, \mathbf{m}_{s},  r_{s})\in \mathrm{R}_+^{|\mathcal{L}_{s}|}\times\mathrm{R}^{\ell+1}\times   \mathrm{R}_+:\ {m}_{s,0}=1, \\
& \mathbf{M}(0, \ell, \mathbf{m}_{s})
\succeq 0, \ \beta_{s}  \mathbf{M}(0, \ell-2, \mathbf{m}_{s})-  \mathbf{M}(2,\ell, \mathbf{m}_{s}) \succeq 0,\\
& x_{s,l}^{\text{out}}\leq c_l, l\in \mathcal{L}_{s}, m_{s, j} \leq (r_{s})^{j/\ell},\ j\in \{1,\ldots,\ell\},\ ( \mathbf{x}_{s}^{\text{out}},\ r_{s})\in \mathcal{X}_{s}\}.
\end{aligned}
}
\end{equation}

\subsection{Algorithm DPDA}
The constrains set of convex relaxation (\ref{op_five}) consists of local constraints, e.g., capacity constraints and global constraints, e.g., flow conservation constraints through nodes. The existence of global constraints renders difficulty for us to solve problem (\ref{op_five}) in a distributed fashion. However, 
the primal-dual method, proposed by Chambolle and Pock in \cite{Chambolle_mathe_2015} for solving convex-concave saddle point problems makes it possible. This algorithm can be adapted to solve the multi-agent consensus optimization problem as discussed in \cite{Aybat_arxiv_2016}. We also use the distributed primal-dual algorithm in \cite{Aybat_arxiv_2016} to solve our traffic allocation problem (\ref{op_five}). We present the resulting iterative algorithm, i.e., DPDA, of which iterate sequence converges to the solution of (\ref{op_five}). The details of developing DPDA can be found in Appendix B.

\begin{algorithm}\label{algorithm}
\caption{\textbf{DPDA}  $\newline\gamma, [\kappa_{b,l}]_{l\in\mathcal{L}_b,b\in \mathcal{B}}, [\tau_{s}]_{s\in\mathcal{S}},[\tau_{i,b,l}]_{i\in \mathcal{I}_{b,l}^{\text{out}},l\in\mathcal{L}_b,b\in \mathcal{B}}, [\lambda_{b,l}^0]_{l\in\mathcal{L}_b,b\in \mathcal{B}}, \newline \mathbf{x}_{s}^{{\text{out}},0}=[x_{s,l}^{{\text{out}},0}]_{l\in \mathcal{L}_{s}},
r_{s}^0, \mathbf{m}_{s}^0,\newline\mathbf{x}_{b,l}^{{\text{out}},0}=[x_{i,b,l}^{{\text{out}},0}]_{i\in \mathcal{I}_{b,l}^{\text{out}},l\in\mathcal{L}_b, b\in \mathcal{B}}$}
 Initialization $z_{s,l}^0 \leftarrow x_{s,l}^{{\text{out}},0}$, $\forall l\in \mathcal{L}_{s}$, $ s\in \mathcal{S}$, $z_{i,b,l}^0 \leftarrow x_{i,b,l}^{{\text{out}},0}$, $\forall i \in \mathcal{I}_{b,l}^{\text{out}}$, $l \in \mathcal{L}_b$, $b\in \mathcal{B}$\\
\For{$k\geq 0$}
{/* Each source node $ s \in \mathcal{S}$ updates its desired rate by solving a convex semidefinite program.*/
$(\mathbf{x}_{s}^{{\text{out}}, k+1},\ \mathbf{m}_{s}^{ k+1},\  r_{s}^{k+1})
\leftarrow
\mathit{P}_{\mathcal{A}_{s}}\bigg( [x_{s,l}^{{\text{out}},k}-\gamma \tau_{s} (z_{s,l}^k-  \sum\limits_{\tilde{l} \in \mathcal{L}_{e_l(s)}} z_{i,e_l(s), \tilde{l}}^k)]_{l\in \mathcal{L}_{s}},\ \mathbf{m}_{s}^k +\tau_{s}\mathbf{p}_{s}  ,\ r_{s}^{k} \bigg)$\\
/*Each forwarding node $b\in\mathcal{B}$ updates its desired sending data rate.*/
$
{x}_{i,b,l}^{{\text{out}}, k+1}
\leftarrow
 \mathit{P}_{\mathrm{R}_{+}} \bigg(
 {x}_{i,b,l}^{{\text{out}}, k} - \tau_{i,b,l}(  \lambda_{b,l}^k+ \gamma  (z_{i, b, l}^k+u_{i, b, l}^k-u_{i, e_l(b), l}^k
 )\bigg),\ \forall i\in \mathcal{I}_{b,l}^{\text{out}},\ l \in \mathcal{L}_b,\  b\in \mathcal{B}$,
  where
 $
 u_{i, b, l}^k=\sum\limits_{\tilde{l}\in \mathcal{L}_{b,i}^{\text{out}}}z_{i, b, \tilde{l}}^k-  \sum\limits_{\bar{l} \in \mathcal{L}_{b,i}^{\text{in}}} z_{i, e_{\bar{l}}(b), \bar{l}} ^k   $
  and
   $
  u_{i, e_l(b), l}^k=\sum\limits_{\hat{l}\in \mathcal{L}_{e_l(b),i}^{\text{out}}}z_{i, e_l(b), \hat{l}}^k-  \sum\limits_{\breve{l} \in \mathcal{L}_{e_l(b),i}^{\text{in}}} z_{i, e_{\breve{l}}(e_{l}(b)), \breve{l}} ^k  $\\
/*Each link $l\in \mathcal{L}$ updates its link price.*/
$\newline \lambda_{b, l}^{k+1}
\leftarrow  \mathit{P}_{\mathrm{R}_+}\bigg(\lambda_{b, l}^k+\kappa_{b, l} (\mathbf{1}_{b,l}(2 \mathbf{x}_{b,l}^{{\text{out}}, k+1}-\mathbf{x}_{b,l}^{{\text{out}}, k})$
$+
\delta_{b,l}(2 \mathbf{x}_{e_l(b),l}^{{\text{out}}, k+1}-\mathbf{x}_{e_l(b),l}^{{\text{out}}, k}) - c_l)\bigg),\ \forall b\in \mathcal{B} $\\
/*The following local variables are communicated among neighboring nodes.*/
$z_{s,l}^{k+1}
\leftarrow z_{s,l}^{k} - x_{s,l}^{{\text{out}}, k}+ 2x_{s,l}^{{\text{out}}, k+1},\ \forall l\in \mathcal{L}_{s},\ s\in \mathcal{S}
\newline z_{i, b, l}^{k+1}
\leftarrow z_{i, b, l}^{k}- x_{i, b, l}^{{\text{out}}, k}+2 x_{i, b, l}^{{\text{out}}, k+1},\ \forall i \in \mathcal{I}_{b,l}^{\text{out}},\ l \in \mathcal{L}_b,\ b\in \mathcal{B}$
}
\end{algorithm}

The suboptimality and feasibility of the DPDA
iterate sequence can be bounded as in the following theorem.

\begin{theorem}\label{theorem_convergence}
Given the communication network and the convex optimization problem (\ref{op_five}). Let $d_{s}>0,\ s\in \mathcal{S}$ and $d_{i,b,l}>0,\ i\in \mathcal{I}_{b,l}^{\text{out}},\ l\in \mathcal{L}_b,\ b\in \mathcal{B}$ be given (sufficiently large) constants. Recall that the decision variable $\mathbf{x}$ of problem (\ref{op_five}) is a vector consisting of the data rate $\mathbf{x}_{s,l}^{\text{out}}$, $r_{s}$ and $\mathbf{m}_{s}$ for each $s\in\mathcal{S}$, and the sending data rate
$\mathbf{x}_{b,l}^{\text{out}},\ b\in\mathcal{B} $ for each $l\in \mathcal{L} $. Also recall that vector variables $\mathbf{\lambda},\mathbf{\theta}$ are the dual variables associated with the capacity constraints and the flow conservation constraints at nodes, respectively. Let $(\mathbf{x}^{\star},\mathbf{\lambda}^{\star},\mathbf{\theta}^{\star})$ be an arbitrary saddle-point for the Lagrange function of problem (\ref{op_five}), and
$\{\mathbf{x}^{k}\}_{k\geq 0}$ be the iterate sequence generated using Algorithm DPDA, initialized from an arbitrary $\mathbf{x}^0$ and $[\lambda_{b,l}^0]_{l\in\mathcal{L}_b,b\in \mathcal{B}}=\mathbf{0}$. 
 Let the primal-dual step sizes $[\tau_{s}]_{s\in\mathcal{S}},[\tau_{i,b,l}]_{i\in \mathcal{I}_{b,l}^{\text{out}},l\in\mathcal{L}_b,b\in \mathcal{B}}$ and $\gamma$ be positive constants satisfying the following inequalities
 \begin{equation}\label{condition_one}
\frac{1}{\tau_{s}}-\gamma(4+{d}_{s} ) \geq 0,
\end{equation}
for all $\ s\in \mathcal{S}$, and
 \begin{equation}\label{condition_two}
\frac{1}{\kappa_{b,\ l}}\bigg(\frac{1}{\tau_{i,b,l}}-\gamma(4+{d}_{i,b,l} ) \bigg)\geq m_l+1,
\end{equation}
for all $i\in\mathcal{I}_{b,l}^{\text{out}},l\in\mathcal{L}_b,b\in\mathcal{N}$, where $m_l$ is the total number of sources using link $l$ to transport flows. Denote the average of sending data rates
by $\bar{\mathbf{x}}^{K}\triangleq \frac{1}{K} \sum\limits_{k=1}^K \mathbf{x}^k$, where $K\geq 1$. Then, $\{\bar{\mathbf{x}}^{K}\}$ converges to the maximum of the utility function of the problem (\ref{op_five}) subject to the resource allocation constraints. In particular, the average of the iterative sequence asymptotically converges to the feasible solution, i.e.,
\begin{equation}\label{theorem_error_bounds_one}
\|\mathbf{\theta}^{\star}\| \|B\bar{\mathbf{x}}^{K} \|+  \sum\limits_{b\in\mathcal{B}}\sum \limits_{l\in\mathcal{L}_b} \| \mathbf{\lambda}^{\star}_{b,l} \| h(\bar{\mathbf{x}}_{b,l}^{\text{out}},\bar{\mathbf{x}}_{e_l(b),l}^{\text{out}})
\leq \frac{\Theta_1}{K}, \forall K\geq 1.
\end{equation}
It also asymptotically maximizes the utility function of the problem (\ref{op_five}), i.e.,
\begin{equation}\label{theorem_error_bounds_two}
| \sum\limits_{s\in\mathcal{S}}\mathbf{p}_s^T(\bar{\mathbf{m}}_s-\mathbf{m}_s^{\star})|
\leq  \frac{\Theta_1}{K}, \forall K\geq 1,
\end{equation}
where the notation $h(\bar{\mathbf{x}}_{b,l}^{\text{out}},\bar{\mathbf{x}}_{e_l(b),l}^{\text{out}})$ and $\Theta_1$ is defined in Appendix C.

\end{theorem}

\begin{proof}
The proof is presented in Appendix C.
\end{proof}

Algorithm DPDA is a fully distributed traffic allocation algorithm. 
This point can be verified by looking through the implementation procedure. 
The step-size parameters are decided before implementing the algorithm. It is given in Theorem \ref{theorem_convergence} that those parameters  satisfy conditions (\ref{condition_one}) and (\ref{condition_two}), both of which are local conditions. Thus, choosing the parameters requires no global information. In the first step, the variables $z_{s,l},\ l\in \mathcal{L}_{s},\ s\in \mathcal{S}$ and $z_{i,b,l},\ i \in \mathcal{I}_{b,l}^{\text{out}},\ l \in \mathcal{L}_b,\ b\in \mathcal{B}$ are local variables respectively introduced for each source node and each forwarding node. It is worth noting that giving the initial state value of $x_{s,l}^{\text{out}},\ l\in \mathcal{L}_{s},\ s\in \mathcal{S}$ and $x_{i,b,l}^{\text{out}},\ i \in \mathcal{I}_{b,l}^{\text{out}},\ l \in \mathcal{L}_b,\ b\in \mathcal{B}$ to those introduced variables is also a local operation. For the first iteration, i.e., $K=1$, in steps $3$ and $4$, DPDA enables all nodes to update their sending data rates in parallel. Each node solely uses immediate information from its neighboring nodes to perform all computations. In step $5$, the link price $\lambda_{b, l}^{k+1},\ l \in \mathcal{L}_b,\  b\in \mathcal{B}$ is updated with new local data rate allocation solution. This step can be performed at both end points that each link connects, which just uses their local information. Step $6$
updates the introduced local variables with the new local data rate allocation solution. The iterative procedure continues until the iterate sequence converges to the optimal solution.

\begin{remark}
It follows from inequalities (\ref{theorem_error_bounds_one}) and (\ref{theorem_error_bounds_two}) that DPDA converges at the rate of $\mathit{O}(1/K)$, where $K$ is the number of iterations.

\end{remark}

\begin{remark}
If the problem (\ref{op_five}) has a unique solution, then the sequence of sample averages converges to that solution.
\end{remark}

\section{SIMULATION RESULTS}
In this section, we present some simulation results which exemplify the behavior of the proposed algorithm, i.e., Algorithm DPDA. The simulations show that the final data rate allocation results in a value of the utility function barely distinguishable from the optimal one.

We consider the network model shown in Fig.~2, where we also show all the links' bandwidths, and source-destination pairs. The network model allows for multiple paths available for flows belonging to each source node. We consider a total of $8$ different combinations of source/destination
nodes. Moreover, we list the prescribed next hops for all forwarding nodes $b_i,\ i=1,\ldots,8$, in Table II. For example, the upper left cell means that node $b_1$ forwards the data of source $s_{1}$ to nodes $b_{2}$ and $b_{7}$.

The objective throughout the simulation is to maximize the sum utility of source nodes, where source $s_i,\ i=1,\ldots,8$, has the utility function given by
\[
{\small
\begin{aligned}
U_{s_i}(r_{s_i})
=&  1.763(r_{s_i})^{1/6}
-20.718(r_{s_i})^{2/6}
+88.568(r_{s_i})^{3/6}\\
&-169.102(r_{s_i})^{4/6}
+145.167(r_{s_i})^{5/6}-44.677(r_{s_i})^{6/6}.
\end{aligned}
}
\]
$U_{s_i}(r_{s_i})$ is
a step-like non-concave polynomial-like function. We consider to optimize a step-like non-concave function, because it is more likely to describe the video quality perceived by a user in a video streaming application \cite{nekouei}.
Moreover, we obtain the resource constraints information from Fig.~2 and Table.~II, and impose the lower and upper bounds on the aggregate data rate of each user as $\xi_{s_i}=0$ and $\zeta_{s_i}=10,\ i=1,\ldots,8$, respectively.

Given the network topology shown in Fig.~2, we choose the step-size parameters to satisfy the convergence condition set forth by Theorem $2$. All step-size parameters are chosen locally using local information.
Fig.~3 shows the performance of Algorithm DPDA for these step-size parameters. It can be seen that the utility function converges to the optimal one, which is obtained by using Genetic Algorithm while assuming the availability of global information. Although all the computations of DPDA are performed locally at each node, it attains almost the same network utility obtained by a centralized optimization algorithm. This implies that the iterate sequence of Algorithm DPDA can indeed converge to the optimal traffic allocation.

Fig.~4 shows the representative data rate trajectories for MRGUBS flows belonging to source nodes $s_3$
and $s_4$. Both data rate sequences are generated by DPDA. It can be seen from Fig.~4 that the MRGUBS requirements are satisfied.

\begin{figure}
\centering
\includegraphics[width=1\columnwidth,height=0.45 \columnwidth]{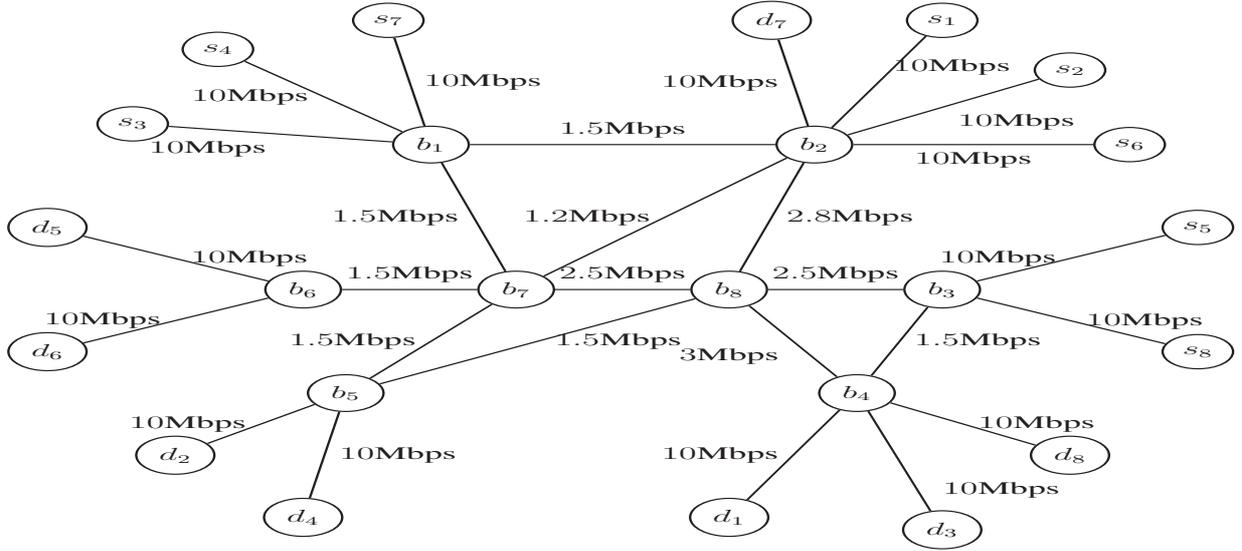} \\
\caption{Topology of the communication network.}
\label{fig_3}
\end{figure}

\begin{table}
\centering
\caption{ROUTING DECISIONS BY SOURCE NODES OF FLOWS}
\label{table:routing}
\resizebox{0.7\textwidth}{!}{
 \begin{tabular}{| c | l | l | l | l | l | l | l | l |}
 \hline
  & $b_1$ & $b_2$ & $b_3$ & $b_4$ & $b_5$ & $b_6$ & $b_7$ & $b_8$\\
 \hline
   $s_1$ & $b_2, b_7$ &$b_7, b_8$  & $b_4$ &$d_1$ & --& --&$b_8$ & $b_3, b_4$\\[2pt]
  \hline
   $s_2$ & $b_2, b_7$ &$b_7, b_8$  & -- &-- & $d_2$& --&$b_5$ & $b_5, b_7$\\[2pt]
 \hline
 $s_3$ & $b_2, b_7$ &$b_7, b_8$  & $b_4$ & $d_3$&-- & --&$b_8$ & $b_3, b_4$\\[2pt]
 \hline
  $s_4$ & $b_2, b_7$ &$b_7, b_8$  &-- & --& $d_4$ & --&$b_5$ & $b_5, b_7$\\[2pt]
 \hline
  $s_5$ & $b_7$ &$b_1,b_7, b_8$  &$b_4, b_8$ & $b_8$& $b_7$ & $d_5$&$b_6$ & $b_5, b_7$\\[2pt]
 \hline
   $s_6$ & $b_7$ &$b_1,b_7, b_8$  &$b_4, b_8$ & $b_8$& $b_7$ & $d_6$&$b_6$ & $b_5, b_7$\\[2pt]
 \hline
   $s_7$ & $b_2,b_7$ &$d_7$ &-- & --& -- &-- &$b_2,b_8$ & $b_2$\\[2pt]
 \hline
 $s_8$& $b_2,b_7$ &$b_7,b_8$ &$b_4$ & $d_8$& -- &-- &$b_8$ & $b_3, b_4$\\[2pt]
 \hline
  \end{tabular}
  }
  \end{table}

\section{CONCLUSIONS AND DIRECTIONS FOR FUTURE RESEARCH}
In this paper, we proposed a distributed traffic allocation algorithm, i.e., DPDA, to allow distributed optimal traffic engineering in a connectionless autonomous network. DPDA is distributed and converges at a $\mathit{O}(1/K)$ rate, where $K$ is the number of iterations. Moreover, numerical simulation results showed that the behavior of DPDA mimics the optimal traffic distribution.

The results presented in this paper are just the first step towards the implementation of an optimal fast distributed algorithm for traffic engineering. There are many issues that need further consideration. In particular, efforts should be put on testing the implementation in large-scale network settings.

\begin{figure}
\centering
\includegraphics[width=1\columnwidth, height=0.3 \columnwidth]{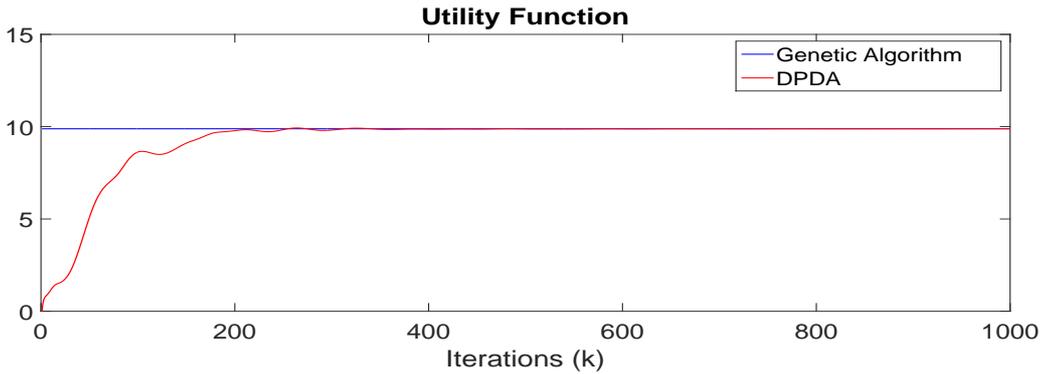} \\
\caption{Value of utility function obtained by DPDA and Genetic Algorithm.}
\label{fig_4}
\end{figure}

\begin{figure}
\centering
\includegraphics[width=1\columnwidth, height=0.3\columnwidth]{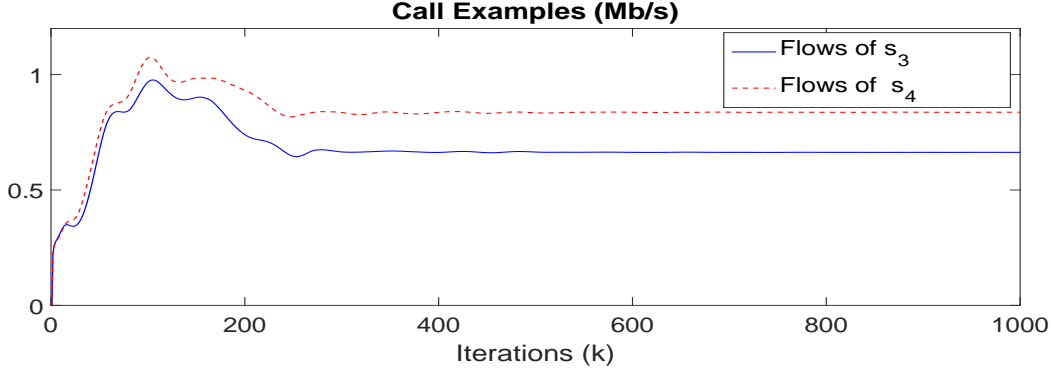} \\
\caption{The data rate trajectory for MRGUBS flows belonging to source nodes $s_3$
and $s_4$.}
\label{fig_5}
\end{figure}

\section*{APPENDIX A. PRELIMINARY RESULTS AND PROOF OF PROPOSITION $1$}
In this Appendix, we include some results from real analysis theory and the main steps of proving Proposition $1$.

\subsection{Preliminary results}
In this subsection, we first recall some results from real analysis theory which are fundamental for the traffic allocation in connectionless networks.
\begin{lemma}\label{Lemma1}
Let $f$ be an arbitrary real-valued function, $\mathcal{F}$ be a compact set, not necessarily convex, and $\mu$ be a probability measure. Then,
\begin{equation}
\underset{x}{\mathrm{inf}} \left\{ f(x) : x \in \mathcal{F} \right\} =
\underset{\mu}{\mathrm{inf}} \left\{ \int f d \mu : \mathrm{supp}(\mu) \subset \mathcal{F} \right\},
\end{equation}
where $\mathrm{supp}(\mu)$ denotes the support of the measure $\mu$.
\end{lemma}
\begin{proof}
For the sake of completeness, we briefly mention the main steps of this well-known fact. Let $x^{\star} \in \mathcal{F}$ be a minimizer of $f$ such that $f(x)\geq f(x^{\star})$ for every $x \in \mathcal{F}$. Then, we have $ \int f d \mu\geq f(x^{\star})$ hold for every probability measure $\mu$ with $\mathrm{supp}(\mu) \subset \mathcal{F}$. That is to say, we have the following inequality hold
\begin{equation}\label{appendix_lemma_condition_one}
\underset{x}{\mathrm{inf}} \left\{ f(x) : x \in \mathcal{F} \right\} \leq
\underset{\mu}{\mathrm{inf}} \left\{ \int f d \mu : \mathrm{supp}(\mu) \subset \mathcal{F} \right\}.
\end{equation}

On the other hand, we have $\int f d \delta_{x^{\star}}= f(x^{\star})$, where $\delta_{x^{\star}}$ is the Dirac measure of $x^{\star}$ on the set $ \mathcal{F}$. Since $\delta_{x^{\star}}$ is a particular probability measure with $\mathrm{supp}(\delta_{x^{\star}}) \subset \mathcal{F}$ and $\int f d \delta_{x^{\star}}= f(x^{\star})$, we have 
\begin{equation}\label{appendix_lemma_condition_two}
\underset{x}{\mathrm{inf}} \left\{ f(x) : x \in \mathcal{F} \right\} \geq
\underset{\mu}{\mathrm{inf}} \left\{ \int f d \mu : \mathrm{supp}(\mu) \subset \mathcal{F} \right\}.
\end{equation}

In conclusion, the result of Lemma\ref{Lemma1} is established by (\ref{appendix_lemma_condition_one}) and (\ref{appendix_lemma_condition_two}).
\end{proof}

We proceed with the following theorem \cite{ozay} that provides necessary and sufficient conditions for the existence of Borel measures whose support is included in bounded symmetric intervals of the real line.

\begin{theorem}\label{Theorem1}
Given a sequence $\mathbf{t}\triangleq\{t_{j}\}_{j=1}^{\ell}$ and a scalar $\epsilon>0$, there exists a Borel measure $\mu(.)$ with support contained in $\mathcal{Y}\doteq [-\epsilon,\epsilon]$ such that $\mu(\mathcal{Y})=1$ and
$t_{j}=E_{\mu}[y^j]=\int_{\mathcal{Y}} y^{j} \mu(dy)$ is true if and only if
\begin{itemize}
\item when $\ell=2k+1$ (odd case), the following holds
\begin{align}
&\epsilon\mathbf{M}(0,2k,\mathbf{t}) \succeq \mathbf{M}(1,2k+1,\mathbf{t})  \\
& \mathbf{M}(1,2k+1,\mathbf{t}) \succeq -\epsilon \mathbf{M}(0,2k,\mathbf{t}),
\end{align}

\item when $\ell=2k$ (even case), the following holds
\begin{align}
&\mathbf{M}(0,2k,\mathbf{t}) \succeq 0  \\
&\epsilon^{2} \mathbf{M}(0,2k-2,\mathbf{t}) \succeq \mathbf{M}(2,2k,\mathbf{t}),
\end{align}
\end{itemize}
where $\mathbf{M}(k,k+2h,\mathbf{t}) \in \mathbb{R}^{(h+1) \times (h+1)}$ is a Hankel matrix of the form
\begin{equation}\label{hankel_matrix}
\mathbf{M}(k,k+2h,\mathbf{t}) =
\left[\begin{smallmatrix}
t_{k} & t_{k+1} & \ldots &  t_{k+h}\\
t_{k+1} &  \reflectbox{$\ddots$}  &  \reflectbox{$\ddots$}   & t_{k+h+1}\\
\vdots &   \reflectbox{$\ddots$}   & \reflectbox{$\ddots$}   &  \vdots\\
t_{k+h} & \ldots &\ldots & t_{k+2h}
\end{smallmatrix}\right].
\end{equation}
and $t_{0}=1$.
\end{theorem}

\begin{proof}
Direct application of Theorem $\mathrm{III}.2.3$ and Theorem $\mathrm{III}.2.4$ in \cite{krein}.
\end{proof}

\subsection{Proof of Proposition $1$}
\begin{proof}

We note that the problem can be converted into a  polynomial optimization form with a change of variables $y_{s}=(r_{s})^{(1/\ell)}$.
The equivalent problem is stated as follows.
\begin{equation}\label{op_two}
{\small
\begin{aligned}
&\underset{\mathbf{x}}{\text{maximize}}
& & \sum\limits_{s\in \mathcal{S}} \sum\limits_{j=0}^{\ell} p_{s,j}(y_{s})^{j}\\
&\text{subject\ to} & & y_{s} \leq (r_{s})^{(1/\ell)},\  s \in \mathcal{S},\\
& & &\sum\limits_{i \in \mathcal{I}_{b,l}^{\text{out}}} x_{i,b,l}^{\text{out}}+\sum\limits_{i \in \mathcal{I}_{b,l}^{\text{in}}}x_{i,e_l(b),l}^{\text{out}} \leq c_l,\ l \in \mathcal{L}_b,\ b \in \mathcal{B},\\
& & & \sum\limits_{l \in \mathcal{L}_{b,i}^{\text{out}}} x_{i,b,l}^{\text{out}}- \sum\limits_{\tilde{l} \in \mathcal{L}_{b}}x_{i,e_{\tilde{l}}(b),\tilde{l}}^{\text{out}}=0,\ i \in \mathcal{I}_b,\  b \in \mathcal{B},\\
& & & ( \mathbf{x}_{s}^{\text{out}},\ r_{s})\in \mathcal{X}_{s},\  s \in \mathcal{S},\\
& & & x_{i,b,l}^{\text{out}}\geq 0,\ i\in \mathcal{I}_{b,l}^{\text{out}},\ l\in \mathcal{L}_b,\  b\in \mathcal{B}.
\end{aligned}
}
\end{equation}
Note that the feasible set in (\ref{op_two}) is convex.
However, the equivalent problem is still a non-convex problem, because of the non-concavity of the utility function. Then, instead of working with $y_{s}$, we optimize over moments of probability distributions in the space of $y_{s}$. More precisely, suppose $y_{s}$ is a random variable and we denote by $m_{s,j}$ the $j$-th moment of $y_{s}$ for some probability measure $\mu$, i.e., $m_{s,j}=E_{\mu}[y_{s}^j]$.

Now, we consider transforming problem (\ref{op_two}) into an optimization problem over the space of probability measures of $y_{s}$ with a support contained in the feasible set of
(\ref{op_two}). 
\begin{enumerate}
  \item Based on Lemma $1$, the objective function becomes
\begin{equation}
{\small
\int \displaystyle \sum\limits_{s\in \mathcal{S}} \sum\limits_{j=0}^{\ell} p_{s,j}(y_{s})^{j} d \mu_{i}
= \sum_{s \in \mathcal{S}} \mathbf{p}_{s}^{T} \mathbf{m}_{s}.
}
\end{equation}
  \item  The first three constraints in (\ref{op_five}) are justified by Theorem $1$.
  \item We use the set of constraints
\begin{equation}\label{approx_constraints}
{\small
m_{s,j} \leq r_{s}^{j / \ell}, ~ j \in \{1,\hdots,\alpha\}
}
\end{equation}
to approximate the constraint $y_{s} \leq (r_{s})^{(1/\ell)}$.
  \item The left hand of each constraint $\sum\limits_{i \in \mathcal{I}_{b,l}^{\text{out}}} x_{i,b,l}^{\text{out}}+\sum\limits_{i \in \mathcal{I}_{b,l}^{\text{in}}}x_{i,e_l(b),l}^{\text{out}} \leq c_l $ for $l\in \mathcal{L}_b,\ b\in \mathcal{B}$,
is written as $\mathbf{1}_{b,l} \mathbf{x}_{b,l}^{\text{out}}+\delta_{b,l}  \mathbf{x}_{e_l(b),l}^{\text{out}}$. In a similar way,
we rewrite constrains $\sum\limits_{l \in \mathcal{L}_{b,i}^{\text{out}}} x_{i,b,l}^{\text{out}}- \sum\limits_{\tilde{l} \in \mathcal{L}_{b,i}^{\text{out}}}x_{i,e_{\tilde{l}}(b),\tilde{l}}^{\text{out}}=0,\ i \in \mathcal{I}_b,\  b \in \mathcal{B}$ in a matrix form, i.e., $B\mathbf{x}=\mathbf{0}$.
\end{enumerate}

In conclusion, Lemma $1$, Theorem $1$ and (\ref{approx_constraints}) establish the result of Proposition $1$.
\end{proof}

\section*{APPENDIX B. DERIVATION OF DPDA}
The constrains set of convex relaxation (\ref{op_five}) consists of local constraints, e.g., capacity constraints and global constraints, e.g., flow conservation constraints through nodes. The existence of global constraints renders difficulty for us to solve problem (\ref{op_five}) in a distributed fashion. However, 
the primal-dual method, proposed by Chambolle and Pock in \cite{Chambolle_mathe_2015} for solving convex-concave saddle point problems makes it possible. This algorithm can be adapted to solve the multi-agent consensus optimization problem as discussed in \cite{Aybat_arxiv_2016}. We also use the distributed primal-dual algorithm in \cite{Aybat_arxiv_2016} to solve our traffic allocation problem (\ref{op_five}). This Appendix aims at developing the distributed algorithm that converges to the solution of (\ref{op_five}).

The optimization problem (\ref{op_five}) can be compactly stated as
\begin{equation}\label{op_four}
{\small
\begin{aligned}
& \underset{\mathbf{x}}{\text{maximize}} & & \sum\limits_{s \in \mathcal{S}} \mathbf{p}_{s}^T\mathbf{m}_{s}\\
&\text{subject\ to}
& & \mathbf{1}_{b,l} \mathbf{x}_{b,l}^{\text{out}}+\delta_{b,l} \mathbf{x}_{e_l(b),l}^{\text{out}}-c_l\leq 0,\ l\in \mathcal{L}_{b},\ b \in \mathcal{B}\\
& & &B\mathbf{x}=0,\\
& & &(\mathbf{x}_{s}^{\text{out}},\ \mathbf{m}_{s},\ r_{s})\in \mathcal{A}_{s},\  s\in \mathcal{S},\\
& & &\mathbf{x}_{b,l}^{\text{out}}\succeq 0,\ l\in \mathcal{L}_b,\  b\in \mathcal{B},
\end{aligned}
}
\end{equation}
where $\mathcal{A}_{s}$ are the set of local constraints for each source node $s \in \mathcal{S}$, as defined in (\ref{set_s}).

We introduce the convex-concave saddle-point form of the primal problem (\ref{op_four}),
\begin{equation}\label{saddle_point_problem}
\underset{\mathbf{x}}{\min} \underset{\lambda,\ \theta}{\max}L(\mathbf{x},\ \lambda,\ \theta),
\end{equation}
 where $L(\mathbf{x},\ \lambda,\ \theta)$ is the Lagrangian function given by
\begin{equation}\label{lag_func}
{\small
\begin{aligned}
L(\mathbf{x},\ \lambda,\ \theta)
=&-\sum\limits_{s \in \mathcal{S}} (\mathbf{p}_{s}^T\mathbf{m}_{s}
-I_{\mathcal{A}_{s}}(\mathbf{x}_{s}^{\text{out}},\ \mathbf{m}_{s},\ r_{s})  )\\
&+\sum\limits_{b\in \mathcal{B}} \sum\limits_{l\in \mathcal{L}_b} I_{\mathrm{R}_{+}^{|\mathcal{I}_{b,l}^{\text{out}}|}}(\mathbf{x}_{b,l}^{\text{out}})-\sum\limits_{b\in \mathcal{B}} \sum\limits_{l\in \mathcal{L}_b} I_{{{R}_{+}}}(\lambda_{b,l})
\\
&+\sum\limits_{b\in \mathcal{B}} \sum\limits_{l\in \mathcal{L}_b}\fprod{\mathbf{1}_{b,l} \mathbf{x}_{b,l}^{\text{out}}+\delta_{b,l}\mathbf{x}_{e_l(b),l}^{\text{out}} - c_l,\ \lambda_{b,l}}\\
&+\fprod{B\mathbf{x},\ \theta}.
\end{aligned}
}
\end{equation}
$\theta\in\mathrm{R}^{ \sum\limits_{b\in \mathcal{B}} |I_{b}| }$ is the vector of dual variables associated with the flow conservation constraint at nodes $B\mathbf{x}=0$. Given $l\in \mathcal{L}_b$ and $b\in \mathcal{B}$, the dual variable $\lambda_{b,l}$ is introduced for the capacity inequality constrains $\mathbf{1}_{b,l} \mathbf{x}_{b,l}^{\text{out}}+\delta_{b,l}\mathbf{x}_{e_l(b),l}^{\text{out}} \leq c_l$. Moreover, $\lambda=[\lambda_{b,l}]_{l\in \mathcal{L}_b,b\in \mathcal{B}}$.

Now, given the initial iterates $\mathbf{x}^0$, $\lambda^0$, $\theta^0$ and parameters
$\gamma>0$, $\tau_{s}>0$ for all $s\in\mathcal{S}$, $\tau_{i,b,l}>0$, $\kappa_{b,l}>0$ for all $i\in \mathcal{I}_{b,l}^{\text{out}}$, $l\in \mathcal{L}_b$ and $b\in \mathcal{B}$,
we present the following primal-dual iterations to solve (\ref{saddle_point_problem}):
\begin{equation}\label{x_update_one}
{\small
\begin{aligned}
\mathbf{x}^{k+1}\leftarrow &  \underset{\mathbf{x}}{\text{argmin}}-\sum\limits_{s \in \mathcal{S}} (\mathbf{p}_s^T\mathbf{m}_{s}-I_{\mathcal{A}_{s}}(\mathbf{x}_{s}^{\text{out}},\ \mathbf{m}_{s},\ r_{s})  )\\
&+\sum\limits_{b\in \mathcal{B}} \sum\limits_{l\in \mathcal{L}_b} I_{\mathrm{R}_{+}^{|\mathcal{I}_{b,l}^{\text{out}}|}}(\mathbf{x}_{b,l}^{\text{out}})
\\
&+\sum\limits_{b\in \mathcal{B}} \sum\limits_{l\in \mathcal{L}_b}\fprod{\mathbf{1}_{b,l} \mathbf{x}_{b,l}^{\text{out}}+\delta_{b,l}\mathbf{x}_{e_l(b),l}^{\text{out}} - c_l,\ \lambda_{b,l}^k}\\
&+\fprod{B\mathbf{x},\ \theta^k}+\sum\limits_{b\in \mathcal{B}} \sum\limits_{l\in \mathcal{L}_b} \sum\limits_{i\in \mathcal{I}_{b,l}^{\text{out}}}  \frac{1}{2\tau_{i,b,l}}(x_{i,b,l}^{\text{out}}-x_{i,b,l}^{\text{out},k})^2\\
& + \sum\limits_{s \in \mathcal{S}}\frac{1}{2\tau_{s}}( (x_{s,l}^{\text{out}}- x_{s,l}^{\text{out},k})^2 + \|\mathbf{m}_{s}-\mathbf{m}_{s}^k\|_2^2 +  (r_{s}-r_{s}^k)_2^2   )
;\\
\lambda_{b, l}^{k+1} \leftarrow & \underset{\lambda_{b, l}}{\text{argmax}} -  I_{{\mathrm{R}_{+}}}(\lambda_{b,l})+\langle \mathbf{1}_{b,l}(2 \mathbf{x}_{b,l}^{\text{out}, k+1}-\mathbf{x}_{b,l}^{\text{out}, k})\\
&+
\delta_{b,l}(2 \mathbf{x}_{e_l(b),l}^{\text{out}, k+1}-\mathbf{x}_{e_l(b),l}^{\text{out}, k}) - c_l,\ \lambda_{b, l}\rangle\\
&-\frac{1}{2\kappa_{b, l}}(\lambda_{b, l}-\lambda_{b, l}^k)^2,\  l\in\mathcal{L}_b,\   b\in \mathcal{B};\\
\theta^{k+1} \leftarrow & \underset{\theta}{\text{argmax}} \fprod{B(2\mathbf{x}^{k+1}-\mathbf{x}^k),\ \theta}-\frac{1}{2\gamma} \|\theta-\theta^k\|_2^2\\
& =\theta^k+\gamma B (2\mathbf{x}^{k+1}-\mathbf{x}^k).
\end{aligned}
}
\end{equation}

Although the convergence to the optimal traffic allocation is guaranteed under the primal-dual method, it is still not a distributed algorithm. In fact, solving the optimization problem involved in the primal variables $\mathbf{x}^{k+1}$ update rule requires global information about the network due to the presence of the term  $\fprod{B\mathbf{x},\ \theta^k}$, which is associated with the flow conservation constraints at nodes. Moreover, computing the term $\fprod{\mathbf{1}_{b,l} \mathbf{x}_{b,l}^{\text{out}}+\delta_{b,l}\mathbf{x}_{e_l(b),l}^{\text{out}} - c_l,\ \lambda_{b,l}^k},\  l\in\mathcal{L}_b,\ b\in \mathcal{B}$ forces neighboring nodes to exchange information, because bidirectional links are allowed to exist in the model. This fact hinder us from directly implementing the primal-dual iterations. Nevertheless, we exploit the structure of the inner product $\fprod{B\mathbf{x},\ \theta^k}$ and note that this term is a summation of local linear functions of the local variables. In addition, the sending data rates of neighboring nodes is local information. These observations indicates that it is possible to develop an optimal decentralized traffic allocation algorithm.

Using recursion in $\theta$ update rule in (\ref{x_update_one}), we can write $\theta^{k+1}$ as a partial summation of previous primal variable $\mathbf{x}^k$ iterations, i.e., $\theta^k=\theta^0+\gamma\sum\limits_{n=0}^{k-1} B (2\mathbf{x}^{n+1}-\mathbf{x}^n)$.
Let $\theta^0$ be
$\gamma B \mathbf{x}^{\text{out},0}$, $\mathbf{z}^0$ be $\mathbf{x}^{0}$ and $\mathbf{z}^k \triangleq \mathbf{x}^{k}+\sum\limits_{n=1}^k\mathbf{x}^{n}$ for $k\geq 1$.
Then we get
\begin{equation}\label{x_update_term}
{\small
\begin{aligned}
\fprod{B\mathbf{x},\ \theta^k}
= &\gamma \fprod{ \mathbf{x}^{\text{out}},{B}^T {B} \mathbf{z}^k}\\
=&\gamma \sum\limits_{b\in \mathcal{N}} \sum\limits_{l \in \mathcal{L}_b} \sum \limits_{i \in I_{b,l}^{\text{out}}} x_{i, b, l}^{\text{out}}( \sum \limits_{\tilde{l}\in \mathcal{L}_{b,i}^{\text{out}}} z_{i, b, \tilde{l}}^k-  \sum\limits_{\bar{l} \in \mathcal{L}_{b}} z_{i, e_{\bar{l}}(b), \bar{l}}^k).
\end{aligned}
}
\end{equation}

The quadratic operation for updating $\lambda_{b,l}^{k+1}$ in (\ref{x_update_one}) entails solving the following projection
problem:
\begin{equation}\label{lambda_update_term}
\lambda_{b,l}^{k+1}\leftarrow \mathit{P}_{\mathrm{R}_+}\bigg(\lambda_{b,l}^k+ \mathbf{1}_{b,l}(2 \mathbf{x}_{b,l}^{\text{out},\ k+1}-\mathbf{x}_{b,l}^{\text{out},\ k})
+\delta_{b,l}(2 \mathbf{x}_{e_l(b),l}^{\text{out},\ k+1}-\mathbf{x}_{e_l(b),l}^{\text{out},\ k}) - c_l\bigg).
\end{equation}

Substituting (\ref{x_update_term}) and (\ref{lambda_update_term}) into (\ref{x_update_one}) yields a distributed traffic allocation algorithm shown in Algorithm $1$.

\section*{APPENDIX C. PROOF OF THEOREM $2$}
In this section, we present the Proof of Theorem $2$.

\begin{proof}

Due to space limitations, we only prove that if conditions (\ref{condition_one}) and (\ref{condition_two}) hold, the following inequality is true:
\begin{equation}
Q(A,B)\triangleq \left[ \begin{smallmatrix}
D_{\tau} & -A^T & -B^T\\
-A & D_{\kappa} & \mathbf{0}\\
-B & \mathbf{0} & D_{\gamma}
\end{smallmatrix}\right] \succeq 0
\end{equation}
where
$D_{\kappa} \triangleq \text{diag}([\frac{1}{\kappa_{b,\ l}}]_{ l\in \mathcal{L}_b,\ b\in \mathcal{B}})$, $D_{\gamma} \triangleq \frac{1}{\gamma} I_{\sum \limits_{b\in \mathcal{N}} |I_b|}$, and
$D_{\tau} \triangleq  \text{diag}([v_{s\tau}^T, v_{b\tau}^T]^T )$ where $v_{s\tau}\triangleq [\frac{1}{\tau_{s}} \mathbf{1}_{(|\mathcal{L}_{s}| +\ell+2)\times 1}]_{s\in \mathcal{S}}$ and $v_{b\tau}=[\frac{1}{\tau_{i,b,l}}]_{ i\in\mathcal{I}_{b,l}^{\text{out}},\ l\in \mathcal{L}_b, b\in\mathcal{B}}$. Moreover,
$A\triangleq \mathrm{diag}([A_{l}]_{l\in \mathcal{L}_b,\ b\in \mathcal{B}})$, where $A_{ l}$ is a row vector with the same dimension as vector variable $ \mathbf{x}$, and the $i$-th entry of vector $A_{ l}$, equals to $1$ if the data rate denoted by the $i$-th element is transported through link $l$, $0$ otherwise.

Based on ``Schur complement Lemma", we have ${Q}(A,B) \succeq 0$ holds if and only if
\begin{equation}\label{schur_equivalent_condition}
\left[\begin{smallmatrix}
D_{\tau} & -A^T\\
-A & D_{\kappa}
\end{smallmatrix}\right]
-\gamma
\left[\begin{smallmatrix}
B^TB & \mathbf{0}\\
\mathbf{0} & \mathbf{0}
\end{smallmatrix}\right]\succeq 0.
\end{equation}
Moreover, since $D_{\kappa} \succeq 0$, again using ``Schur complement Lemma", one can conclude that (\ref{schur_equivalent_condition}) holds if and only if
\begin{equation}
D_{\tau}- \gamma B^TB- A^TD_{\kappa}^{-1}A \succeq 0.
\end{equation}
 Denote matrix $B^TB$ by $\Omega$, and we can write
$\Omega$ into the sum of two matrices, i.e.,
\begin{equation}
\Omega=\text{diag}([\omega_{i,b,l}]_{i \in\mathcal{I}_{b,l}^{\text{out}},\ l\in \mathcal{L}_b,\ b\in \mathcal{N}})+E,
\end{equation}
where $\omega_{i,b,l}=1$ if node $b\in \mathcal{S} \bigcup \mathcal{B}_{e}$ where $\mathcal{B}_e$ is the set of nodes that forward traffic to destination nodes, $\omega_{i,b,l}=2$ if $b\in \mathcal{B} \bigcap \mathcal{B}_{e}^c$ where $\mathcal{B}_{e}^c$ is the complement of $\mathcal{B}_{e}$, otherwise $0$. Also, all the diagonal elements of matrix $E$ are equal to $0$, and the  non-diagonal element $E_{(i,b_1,l_1),(j,b_2,l_2)}$, corresponding to data rates $x_{i,b_1,l_1}^\textbf{{out}} \in \mathbf{x}$ and $x_{i,b_2,l_2}^\textbf{{out}} \in \mathbf{x}$, equals to $1$ if both data rates belong to the same source node and they are forwarded from the same node, i.e., $i=j\in\mathcal{I}$ and $b_1=b_2\in\mathcal{N}$, $-1$ if both data rates belong to the same source node and nodes $b_1$ and $b_2$ are neighboring, and $0$ otherwise. Based on ``Gershgorin Circle Theorem" \cite{Golub:1996:MC:248979}, we have
\begin{equation}
\text{diag}([\omega_{i,b,l}]_{ i \in\mathcal{I}_{b,l}^{\text{out}},\ l\in \mathcal{L}_b,\ b\in \mathcal{N}})+
\text{diag}([d_{i,b,l}]_{ i \in\mathcal{I}_{b,l}^{\text{out}},\ l\in \mathcal{L}_b,\ b\in \mathcal{N}})
 -E \succeq 0,
\end{equation} since $d_{i,b,l}$ is chosen to be large enough. Therefore,
\begin{equation}
\Omega \preceq  2\text{diag}([\omega_{i,b,l}]_{ i \in\mathcal{I}_{b,l}^{\text{out}},\ l\in \mathcal{L}_b,\ b\in \mathcal{N}})+
\text{diag}([d_{i,b,l}]_{ i \in\mathcal{I}_{b,l}^{\text{out}},\ l\in \mathcal{L}_b,\ b\in \mathcal{N}})
\end{equation}
Moreover,
\begin{equation}
\Omega \preceq \text{diag}([4+d_{i,b,l}]_{ i \in\mathcal{I}_{b,l}^{\text{out}},\ l\in \mathcal{L}_b,\ b\in \mathcal{N}}).
\end{equation}
Hence, it is sufficient to have
\begin{equation}
{\tau}- \gamma
\text{diag}([4+d_{i,b,l}]_{ i \in\mathcal{I}_{b,l}^{\text{out}},\ l\in \mathcal{L}_b,\ b\in \mathcal{N}}) - A^TD_{\kappa}^{-1}A \succeq 0 ,
\end{equation}
 and this condition holds if the inequalities (\ref{condition_one}) and (\ref{condition_two}) in the statement of Theorem $1$ are true.

Let $(\mathbf{x}^{\star},\mathbf{\lambda}^{\star},\mathbf{\theta}^{\star})$ be an arbitrary saddle-point for the Lagrange function of problem (\ref{op_five}), and
$\{\mathbf{x}^{k}\}_{k\geq 0}$ be the iterate sequence generated using Algorithm DPDA, initialized from an arbitrary $\mathbf{x}^0$ and $[\lambda_{b,l}^0]_{l\in\mathcal{L}_b,b\in \mathcal{B}}=\mathbf{0}$.
Denote the average of sending data rates
by $\bar{\mathbf{x}}^{K}\triangleq \frac{1}{K} \sum\limits_{k=1}^K \mathbf{x}^k$, where $K\geq 1$.
Then, following the proof in \cite{Aybat_arxiv_2016}, we have that $\{\bar{\mathbf{x}}^{K}\}$ converges to the maximum of the utility function of the problem (\ref{op_five}) subject to the resource allocation constraints. In particular, the following error bounds hold for all $K\geq 1$:
\begin{equation}\label{theorem_error_bounds_three}
\begin{aligned}
\|\mathbf{\theta}^{\star}\| \|B\bar{\mathbf{x}}^{K} \|+ & \sum\limits_{b\in\mathcal{B}}\sum \limits_{l\in\mathcal{L}_b} \| \mathbf{\lambda}^{\star}_{b,l} \| h(\bar{\mathbf{x}}_{b,l}^{\text{out}},\bar{\mathbf{x}}_{e_l(b),l}^{\text{out}})
\leq \frac{\Theta_1}{K},\\
&|  \sum\limits_{s\in\mathcal{S}}\mathbf{p}_s^T(\bar{\mathbf{m}}_s-\mathbf{m}_s^{\star})|
\leq  \frac{\Theta_1}{K},
\end{aligned}
\end{equation}
where $h(\bar{\mathbf{x}}_{b,l}^{\text{out}},\bar{\mathbf{x}}_{e_l(b),l}^{\text{out}})$ denotes the distance function $d_{\mathrm{R}_{-}}
(\mathbf{1}_{b,l} \bar{\mathbf{x}}_{b,l}^{\text{out}} + \delta_{b,l} \bar{\mathbf{x}}_{e_l(b),l}^{\text{out}}-c_l)$, and $\Theta_1\triangleq \frac{2}{\gamma}\|\mathbf{\theta}^{\star}\| ^2-\frac{\gamma}{2}   \|B\bar{\mathbf{x}}^{0} \|^2
+ \sum_{b\in\mathcal{B}} \sum_{l\in\mathcal{L}_b} (\sum_{i\in\mathcal{I}_{b,l}^{\text{out}}}  \frac{1}{2\tau_{i,b,l}}(x_{i,b,l}^{{\text{out}},\star}-x_{i,b,l}^{{\text{out}},0})^2+ \frac{1}{2\kappa_{b,l}} (\mathbf{\lambda}^{\star}_{b,l})^2 ) + \sum_{s\in\mathcal{S}} \frac{1}{2\tau_{s}}(\| \mathbf{m}_{s}^{\star}-\mathbf{m}_{s}^0 \|^2+(r_{s}^{\star}-r_{s}^0  )^2+ \sum_{l\in \mathcal{L}_{s}}(x_{s,l}^{{\text{out}},\star}-x_{s,l}^{{\text{out}},0})^2 )$.
\end{proof}

\end{document}